\documentclass[11pt]{amsart}

\newtheorem{theorem}{Theorem}[section]
\newtheorem{lem}[theorem]{Lemma}

\newtheorem{cor}[theorem]{Corollary}

\theoremstyle{definition}
\newtheorem{definition}[theorem]{Definition}

\theoremstyle{remark}
\newtheorem{remark}[theorem]{Remark}

\numberwithin{equation}{section}

\begin{document}

\newcommand{\spacing}[1]{\renewcommand{\baselinestretch}{#1}\large\normalsize}
\spacing{1.14}

\title{Geodesic Vector fields of invariant $(\alpha,\beta)$-metrics on Homogeneous spaces}

\author{M. Parhizkar}

\address{Department of Mathematics\\ Shahrood University of Technology\\ Shahrood\\ Iran.} \email{m\_parhizkar66@yahoo.com}

\author {H. R. Salimi Moghaddam}

\address{Department of Mathematics\\ Faculty of  Sciences\\ University of Isfahan\\ Isfahan\\ 81746-73441-Iran.} \email{hr.salimi@sci.ui.ac.ir and salimi.moghaddam@gmail.com}


%%\date{\today}

\begin{abstract}
In this paper we show that for an invariant $(\alpha,\beta)-$metric $F$ on a homogeneous Finsler manifold $\frac{G}{H}$, induced by an invariant
Riemannian metric $\tilde{a}$ and an invariant vector field $\tilde{X}$, the vector $X=\tilde{X}(H)$ is a geodesic vector of $F$ if and only if it is
a geodesic vector of $\tilde{a}$. Then we give some conditions such that under them, an arbitrary vector is a geodesic vector of $F$ if and only if it is a geodesic vector of $\tilde{a}$. Finally we give an explicit formula for the flag curvature of bi-invariant $(\alpha,\beta)-$metrics on connected Lie groups.
\end{abstract}

\maketitle

%%---------------------------INTRODUCTION--------------------------

\section{\textbf{Introduction}}
The geometry of invariant Finsler structures on homogeneous manifolds is one of the
interesting subjects in Finsler geometry which has been  studied by some Finsler geometers, during recent years(for example see \cite{AnDeng}, \cite{DeHo}, \cite{Latifi1}, \cite{Latifi2}, \cite{Latifi3}, \cite{Sa}.).
An important family of Finsler metrics is the family of
$(\alpha,\beta)-$metrics. These metrics are introduced by M.
Matsumoto (see \cite{Ma}). On the other hand, physicists are also interested in these metrics. They seek
for some non-Riemannian models for spacetime. For example, by using $(\alpha,\beta)-$metrics,
G. S. Asanov introduced Finsleroid–Finsler spaces and formulated pseudo-Finsleroid gravitational field equations (see \cite{Asan1} and \cite{Asan2}.).\\
In the present paper we study geodesic vectors of invariant $(\alpha,\beta)-$metrics on a homogeneous Finsler manifold $\frac{G}{H}$, induced by an invariant Riemannian metric $\tilde{a}$ and invariant vector field $\tilde{X}$. We show that the vector $X=\tilde{X}(H)$ is a geodesic vector of $(\alpha,\beta)-$metric if and only if it is a geodesic vector of $\tilde{a}$. Then we give some conditions such that under them, an arbitrary vector is a geodesic vector of $F$ if and only if it is a geodesic vector of $\tilde{a}$. Finally we give an explicit formula for the flag curvature of bi-invariant $(\alpha,\beta)-$metrics on connected Lie groups.\\
Now we give some preliminaries of Finsler geometry.\\
Let $M$ be a smooth $n-$dimensional manifold and $TM$ be its
tangent bundle. A Finsler metric on $M$ is a non-negative function
$F:TM\longrightarrow \Bbb{R}$ which has the following properties:
\begin{enumerate}
    \item $F$ is smooth on the slit tangent bundle
    $TM^0:=TM\setminus\{0\}$,
    \item $F(x,\lambda y)=\lambda F(x,y)$ for any $x\in M$, $y\in T_xM$ and $\lambda
    >0$,
    \item the $n\times n$ Hessian matrix $[g_{ij}(x,y)]=[\frac{1}{2}\frac{\partial^2 F^2}{\partial y^i\partial
    y^j}]$ is positive definite at every point $(x,y)\in TM^0$.
\end{enumerate}
One of the important quantities which associates with a Finsler
manifold is flag curvature which is a generalization of sectional
curvature to Finsler manifolds. Flag curvature is defined as
follows.

\begin{eqnarray}\label{flag}
  K(P,y)=\frac{g_y(R(u,y)y,u)}{g_y(y,y).g_y(u,u)-g_y^2(y,u)},
\end{eqnarray}
where $g_y(u,v)=\frac{1}{2}\frac{\partial^2}{\partial s\partial
t}(F^2(y+su+tv))|_{s=t=0}$, $P=span\{u,y\}$,
$R(u,y)y=\nabla_u\nabla_yy-\nabla_y\nabla_uy-\nabla_{[u,y]}y$ and
$\nabla$ is the Chern connection induced by $F$ (see \cite{BaChSh}
and \cite{Sh1}.).\\

\begin{definition}
Let $\alpha=\sqrt{\tilde{a}_{ij}(x)y^iy^j}$ be a Riemannian metric and $\beta(x,y)=b_i(x)y^i$ be a $1-$form on an $n-$dimensional manifold $M$. Let
\begin{equation}\label{alpha-norm}
    \|\beta(x)\|_\alpha:=\sqrt{\tilde{a}^{ij}(x)b_i(x)b_j(x)}.
\end{equation}
Now, let the function $F$ be defined as follows
\begin{equation}\label{alpha-beta metric}
    F:=\alpha\phi(s) \ \ \ , \ \ \ s=\frac{\beta}{\alpha},
\end{equation}
where $\phi=\phi(s)$ is a positive $C^\infty$ function on $(-b_0,b_0)$ satisfying
\begin{equation}\label{alpha-beta condition}
    \phi(s)-s\phi'(s)+(b^2-s^2)\phi''(s)>0 \ \ \ , \ \ \ |s|\leq b <b_0.
\end{equation}
Then by Lemma 1.1.2 of \cite{ChSh}, $F$ is a Finsler metric if $\|\beta(x)\|_\alpha<b_0$ for any $x\in M$.
A Finsler metric in the form (\ref{alpha-beta metric}) is called an $(\alpha,\beta)-$metric.
\end{definition}
The Riemannian metric $\tilde{a}$ induces an inner
product on any cotangent space $T^\ast_xM$ such that
$<dx^i(x),dx^j(x)>=\tilde{a}^{ij}(x)$. The induced inner product on
$T^\ast_xM$ induce a linear isomorphism between $T^\ast_xM$ and
$T_xM$. Then the 1-form $b$ corresponds to a vector field
$\tilde{X}$ on $M$ such that
\begin{eqnarray}
  \tilde{a}(y,\tilde{X}(x))=\beta(x,y).
\end{eqnarray}
Also we have $\|\beta(x)\|_{\alpha}=\|\tilde{X}(x)\|_{\alpha}$
(for more details see \cite{DeHo} and \cite{Sa}.).
Therefore we can write $(\alpha,\beta)-$metrics
as follows:
\begin{eqnarray}\label{invariant alpha-beta metric}
  F(x,y)=\alpha(x,y)\phi(\frac{\tilde{a}(\tilde{X}(x),y)}{\alpha(x,y)}),
\end{eqnarray}
where for any $x\in M$,
$\sqrt{\tilde{a}(\tilde{X}(x),\tilde{X}(x))}=\|\tilde{X}(x)\|_{\alpha}<b_0$.\\

\section{\textbf{Geodesic Vector fields of invariant $(\alpha,\beta)$-metrics on Homogeneous spaces}}

\begin{definition}
A Finsler space $(M,F)$ is called a homogeneous Finsler space if the group of isometries of $(M, F)$,
$I(M, F)$, acts transitively on $M$.
\end{definition}

\begin{remark}
Any homogeneous Finsler manifold $M=\frac{G}{H}$ is a reductive homogeneous space.
\end{remark}

\begin{definition}
For a homogeneous Riemannian manifold $(\frac{G}{H},\tilde{a})$, or a homogeneous Finsler manifold $(\frac{G}{H},F)$, a non-zero vector $X\in\frak{g}$ is called a geodesic vector if the curve  $\gamma(t)=\exp(tZ)(o)$ is a geodesic on $(\frac{G}{H},\tilde{a})$, or on $(\frac{G}{H},F)$, respectively.
\end{definition}
Suppose that $(\frac{G}{H},\tilde{a})$ is a homogeneous Riemannian manifold, and $\frak{g}=\frak{m}\oplus\frak{h}$ is a reductive decomposition. In \cite{KoVa}, it is proved that a vector $X\in\frak{g}$ is a geodesic vector if and only if
\begin{equation}\label{geodesic vector condition Riemannian case}
    a([X,Y]_\frak{m},X_\frak{m})=0 \ \ \ , \ \ \ \forall Y\in\frak{m}.
\end{equation}

In \cite{Latifi1}, D. Latifi has proved a similar theorem for Finslerian case as follows:

\begin{theorem}
A vector $X\in\frak{g}-\{0\}$ is a geodesic vector if and only if
\begin{equation}\label{geodesic vector condition Finsler case}
    g_{X_\frak{m}}(X_\frak{m}, [X, Z]_\frak{m})=0 \ \ \ , \ \ \ \forall Z\in\frak{g}.
\end{equation}
\end{theorem}
Also as a corollary of the above theorem he proved the following corollary:
\begin{cor}
A vector $X\in\frak{g}-\{0\}$ is a geodesic vector if and only if
\begin{equation}\label{geodesic vector condition Finsler case}
    g_{X_\frak{m}}(X_\frak{m}, [X, Z]_\frak{m})=0 \ \ \ , \ \ \ \forall Z\in\frak{m}.
\end{equation}
\end{cor}

\begin{theorem}
Let $(\frac{G}{H},F)$ be a homogeneous Finsler manifold, where $F$ is an invariant $(\alpha,\beta)-$metric defined by the relation (\ref{invariant alpha-beta metric}), an invariant Riemannian metric $\tilde{a}$ and an invariant vector field $\tilde{X}$. Then, $X:=\tilde{X}(H)$ is a geodesic vector of
$(\frac{G}{H},F)$ if and only if $X$ is a geodesic vector of $(\frac{G}{H},\tilde{a})$.
\end{theorem}
\begin{proof}
By using the formula $g_y(u,v)=\frac{1}{2}\frac{\partial^2}{\partial t \partial s}F^2(y+su+tv)|_{s=t=0}$ and some computations, for the $(\alpha,\beta)-$metric $F$ defined by relation (\ref{invariant alpha-beta metric}) we have:
\begin{eqnarray}\label{g_y}
    g_y(u,v)&=& \tilde{a}(u,v)\phi^2(r)+\tilde{a}(y,u)\phi(r)\phi'(r)\Big{(}\frac{\tilde{a}(X,v)}{\sqrt{\tilde{a}(y,y)}}-\frac{\tilde{a}(X,y)\tilde{a}(y,v)}{(\tilde{a}(y,y))^{\frac{3}{2}}}\Big{)} \nonumber\\
            && + \Big{(}(\phi'(r))^2+\phi(r)\phi''(r)\Big{)}\Big{(}\frac{\tilde{a}(X,v)}{\sqrt{\tilde{a}(y,y)}}-\frac{\tilde{a}(X,y)\tilde{a}(y,v)}{(\tilde{a}(y,y))^\frac{3}{2}}\Big{)}\nonumber\\
            && \ \ \ \ \ \ \ \ \times\Big{(}\tilde{a}(X,u)\sqrt{\tilde{a}(y,y)}-\frac{\tilde{a}(y,u)\tilde{a}(X,y)}{\sqrt{\tilde{a}(y,y)}}\Big{)}\\
            && +\frac{\phi(r)\phi'(r)}{\sqrt{\tilde{a}(y,y)}}\Big{(}\tilde{a}(X,u)\tilde{a}(y,v)-\tilde{a}(u,v)\tilde{a}(X,y)\Big{)}\nonumber,
\end{eqnarray}
where $r=\frac{\tilde{a}(X,y)}{\sqrt{\tilde{a}(y,y)}}$.
The equation (\ref{g_y}) shows that for any $Z\in\frak{m}$ we have
\begin{equation}\label{eq1}
    g_{X_\frak{m}}(X_\frak{m},[X,Z]_\frak{m})=\tilde{a}(X_\frak{m},[X,Z]_\frak{m})\phi^2(\sqrt{\tilde{a}(X,X)}).
\end{equation}
On the other hand, for an $(\alpha,\beta)-$metric we know that $\phi>0$. Therefore the equation (\ref{eq1}) completes the proof.
\end{proof}

\begin{theorem}
Let $(\frac{G}{H},F)$ be a homogeneous Finsler manifold, where $F$ is an invariant $(\alpha,\beta)-$metric defined by the relation (\ref{invariant alpha-beta metric}), an invariant Riemannian metric $\tilde{a}$ and an invariant vector field $\tilde{X}$. Let $Y\in\frak{g}-\{0\}$ such that $\phi''(r_\frak{m})\leq0$ and for any $Z\in\frak{m}$, $\tilde{a}(X,[Y,Z]_\frak{m})=0$, where $r_\frak{m}:=\frac{\tilde{a}(X,Y_\frak{m})}{\sqrt{\tilde{a}(Y_\frak{m},Y_\frak{m})}}$. Then $Y$ is a geodesic vector of $(\frac{G}{H},F)$ if and only if $Y$ is a geodesic vector of $(\frac{G}{H},\tilde{a})$.
\end{theorem}
\begin{proof}
By using the relation (\ref{g_y}) and some computations we have
\begin{equation}\label{eq2}
    g_y(y,[y,z])=\tilde{a}(y,[y,z])\Big{(}\phi^2(r)-\phi(r)\phi'(r)r\Big{)}+\tilde{a}(X,[y,z])\Big{(}\phi(r)\phi'(r)\sqrt{\tilde{a}(y,y)}\Big{)}.
\end{equation}
So for any $Z\in\frak{m}$ we have
\begin{equation}\label{eq3}
    g_{Y_\frak{m}}(Y_\frak{m},[Y,Z]_\frak{m})=\tilde{a}(Y_\frak{m},[Y,Z]_\frak{m})
    \Big{(}\phi^2(r_\frak{m})-\phi(r_\frak{m})\phi'(r_\frak{m})r_\frak{m}\Big{)}+\tilde{a}(X,[Y,Z]_\frak{m})
    \Big{(}\phi(r_\frak{m})\phi'(r_\frak{m})\sqrt{\tilde{a}(Y_\frak{m},Y_\frak{m})}\Big{)}.
\end{equation}
But we considered, for any $Z\in\frak{m}$, $\tilde{a}(X,[Y,Z]_\frak{m})=0$ therefore we have
\begin{equation}\label{eq4}
    g_{Y_\frak{m}}(Y_\frak{m},[Y,Z]_\frak{m})=\tilde{a}(Y_\frak{m},[Y,Z]_\frak{m})
    \Big{(}\phi^2(r_\frak{m})-\phi(r_\frak{m})\phi'(r_\frak{m})r_\frak{m}\Big{)}.
\end{equation}
Now relation (\ref{alpha-beta condition}), assumption $\phi''(r_\frak{m})\leq0$ and positivity of $\phi$ together with the relation (\ref{eq4}) will complete the proof.

\end{proof}

\begin{lem}
Let $F$ be a bi-invariant Finsler metric on a connected Lie group. Then for every $0\neq y,z\in\frak{g}$ we have:
\begin{equation}\label{eq5}
    g_y(y,[y,z])=0
\end{equation}
\end{lem}
\begin{proof}
See \cite{Latifi2}.
\end{proof}

\begin{theorem}
Let $G$ be a connected Lie group. Suppose that $F$ is an $(\alpha,\beta)$-metric of Berwald type induced by a bi-invariant Riemannian metric $\tilde{a}$ and a left invariant vector field $\tilde{X}$. Let $\{P,y\}$ be a flag constructed in $(e,y)$, and $\{u,y\}$ be an orthonormal basis for $P$ with respect to the inner product induced by the Riemnnian metric $\tilde{a}$ on $\frak{g}=T_eG$. Then the flag curvature of the flag $\{P,y\}$ will be given by
\begin{eqnarray}
  K(P,y) &=& \frac{1}{4\theta}\Big{\{}\Big{(}\phi^2(r)-\phi(r)\phi'(r)r\Big{)}\|[y,u]\|^2_\alpha\nonumber\\
            && +\Big{(}(\phi')^2(r)+\phi(r)\phi''(r)\Big{)}\tilde{a}(X,u)\tilde{a}([X,y],[u,y]) \Big{\}},
\end{eqnarray}
where $r=\frac{\tilde{a}(X,y)}{\tilde{a}(y,y)}=\tilde{a}(X,y)$ and $\theta=\phi^2(r)\Big{(}\phi^2(r)+\phi(r)\phi''(r)\tilde{a}^2(X,u)-\phi(r)\phi'(r)r\Big{)}$.
\end{theorem}
\begin{proof}
Since the $(\alpha,\beta)$-metric $F$ is of Berwald type, the Levi-Civita connection of $\tilde{a}$ and the Chern connection of $F$ and therefore their curvature tensors coincide. So we have
\begin{equation}\label{curvature tensor}
    R(u,y)y=\frac{1}{4}[y,[u,y]].
\end{equation}
Now by using the relations (\ref{curvature tensor}) and (\ref{g_y}) we have:
\begin{eqnarray}\label{4g_y(R,u)}
  4g_y(R(u,y)y,u) &=& \Big{(}\phi^2(r)-\phi(r)\phi'(r)r\Big{)}\tilde{a}([y,[u,y]],u)\nonumber\\
                    && +\Big{(}\phi(r)\phi'(r)\tilde{a}(X,u)-\big{(}(\phi')^2(r)+\phi(r)\phi''(r)\big{)}\tilde{a}(X,u)r\Big{)}\tilde{a}(y,[y,[u,y]])\\
                    && +\Big{(}(\phi')^2(r)+\phi(r)\phi''(r)\Big{)}\tilde{a}(X,u)\tilde{a}(X,[y,[u,y]]),\nonumber
\end{eqnarray}
\begin{equation}\label{g_y(u,u)}
    g_y(u,u)=\phi^2(r)+\Big{(}(\phi')^2(r)+\phi(r)\phi''(r)\Big{)}\tilde{a}^2(X,u)-\phi(r)\phi'(r)r,
\end{equation}
\begin{equation}\label{g_y(y,y)}
    g_y(y,y)=\phi^2(r),
\end{equation}
and
\begin{equation}\label{g_y(y,u)}
    g_y(y,u)=\phi(r)\phi'(r)\tilde{a}(X,u).
\end{equation}
By substituting equations (\ref{4g_y(R,u)}), (\ref{g_y(u,u)}), (\ref{g_y(y,y)}) and (\ref{g_y(y,u)}) in (\ref{flag}) we get
\begin{eqnarray}\label{eq6}
  K(P,y) &=& \frac{1}{4\theta}\{\Big{(}\phi^2(r)-\phi(r)\phi'(r)r\Big{)}\tilde{a}([y,[u,y]],u)\nonumber\\
            && +\Big{(}\phi(r)\phi'(r)+\big{(}(\phi')^2(r)+\phi(r)\phi''(r)\big{)}r\Big{)}\Big{(}\tilde{a}(X,u)\tilde{a}(y,[y,[u,y]])\Big{)} \\
            && +\Big{(}(\phi')^2(r)+\phi(r)\phi''(r)\Big{)}\tilde{a}(X,u)\tilde{a}(X,[y,[u,y]]) \}.\nonumber
\end{eqnarray}
By using (\ref{eq5}) we have $\tilde{a}(y,[y,[u,y]])=0$. On the other hand, since $\tilde{a}$ is bi-invariant we have
$\tilde{a}([y,[u,z]],z)=\tilde{a}([u,y],[z,y])$ (see \cite{Latifi3}). Substituting these two last equations in (\ref{eq6}) completes the proof.
\end{proof}
As a special case when we consider $\phi(s)=1+s$, the formula given in the above theorem for the
flag curvature is equal to the formula given in corollary 3.4 of \cite{Latifi3} for bi-invariant Randers metrics.
%%-------------------- BIBLIOGRAPHY------------------------

\bibliographystyle{amsplain}

\begin{thebibliography}{9}

\bibitem{AnDeng} H. An and S. Deng, \emph{Invariant $(\alpha,\beta)-$metrics on homogeneous manifolds}, Monatsh. Math. \textbf{154}, (2008), 89-102.
\bibitem{Asan1} G. S. Asanov, \emph{Finsleroid space with angle and scalar product}, Publ. Math. Debrecen \textbf{67} (2005) 209–52.
\bibitem{Asan2} G. S. Asanov, \emph{Finsleroid–Finsler spaces of positive-definite and relativistic type}, Rep. Math. Phys. \textbf{58} (2006) 275-300.
\bibitem{BaChSh} D. Bao, S. S. Chern and Z. Shen, \emph{An Introduction to Riemann-Finsler
Geometry}, (Berlin: Springer) (2000).
\bibitem{ChSh} S. S. Chern and  Z. Shen, \emph{Riemann-Finsler Geometry}, World Scientific, Nankai Tracts in Mathematics - Vol. 6, (2005).
\bibitem{DeHo} S. Deng and Z. Hou, \emph{Invariant Randers Metrics on Homogeneous Riemannian
Manifolds}, J. Phys. A: Math. Gen. \textbf{37} (2004), 4353-4360.
\bibitem{KoVa} O. Kowalski and L. Vanhecke, \emph{Riemannian manifolds with homogeneous geodesics}, Boll. Unione. Mat. Ital. \textbf{5} (1991) 189–246.
\bibitem{Latifi1} D. Latifi, \emph{Homogeneous geodesics in homogeneous Finsler spaces}, J. Geom. Phys. \textbf{57} (2007) 1421-1433.
\bibitem{Latifi2} D. Latifi and A. Razavi, \emph{Bi-invariant Finsler Metrics on Lie Groups}, Australian Journal of Basic and Applied Sciences, \textbf{5} (12) (2011) 507-511.
\bibitem{Latifi3} D. Latifi, \emph{Bi-invariant Randers metrics on Lie groups}, Publ. Math. Debrecen, \textbf{76}(1-2) (2010) 219–226.
\bibitem{Ma} M. Matsumoto, \emph{Theory of Finsler spaces with $(\alpha,\beta)-$metric},
Rep. Math. Phys. \textbf{31} (1992), 43-83.
\bibitem{Sa} H. R. Salimi Moghaddam, \emph{The flag curvature of invariant $(\alpha,\beta)-$metrics of type $\frac{(\alpha+\beta)^2}{\alpha}$}, J. Phys. A: Math. Theor. \textbf{41} (2008) 275206 (6pp).
\bibitem{Sh1} Z. Shen, \emph{Lectures on Finsler Geometry}, (World Scientific) (2001).

\end{thebibliography}

\end{document}